\documentclass[10pt]{amsart}
\usepackage{amscd,amssymb}%,showkeys}
\usepackage[mathscr]{eucal}
\usepackage[english]{babel}
\usepackage[linktocpage,bookmarksopen,bookmarksnumbered]{hyperref}
\usepackage{tikz}
\usetikzlibrary{decorations.pathmorphing}
\usepackage{tabularx}

\input xy
\usepackage[all]{xy}

\usepackage{amsmath}
\usepackage[latin1]{inputenc}%
\usepackage{xcolor}

\newlength{\itemlaenge}

\newtheoremstyle{mytheorem}% name
  {}%      Space above, empty = `usual value'
  {}%      Space below
  {\slshape}% Body font
  {}%         Indent amount (empty = no indent, \parindent = para indent)
  {\scshape}% Thm head font
  {.}%        Punctuation after head
  { }%     Space after thm head: " " = normal interword space;
  {}% Thm head spec

\newtheoremstyle{mydefinition}% name
  {}%      Space above, empty = `usual value'
  {}%      Space below
  {\upshape}% Body font
  {}%         Indent amount (empty = no indent, \parindent = para indent)
  {\scshape}% Thm head font
  {.}%        Punctuation after thm head
  { }%     Space after thm head: " " = normal interword space;
  {}% Thm head spec

\theoremstyle{mytheorem}
\newtheorem{lemma}{Lemma}[section]
\newtheorem{prop}[lemma]{Proposition}
\newtheorem*{prop*}{Proposition}

\newtheorem{cor}[lemma]{Corollary}

\newtheorem{thm}[lemma]{Theorem}

\newtheorem*{thm*}{Theorem}
\theoremstyle{mydefinition}

\newtheorem*{rem*}{Remark}

\newtheorem*{notation*}{Notation}

\newtheorem*{warning*}{Warning}

\newtheorem{defi}[lemma]{Definition}
\newtheorem*{defi*}{Definition}

\newtheorem*{ack}{Acknowledgements}

\newcommand{\bqn}{\begin{equation*}}
\newcommand{\eqn}{\end{equation*}}
\newcommand{\bq}{\begin{equation}}
\newcommand{\eq}{\end{equation}}
\newcommand{\ba}{\begin{aligned}}
\newcommand{\ea}{\end{aligned}}
\newcommand{\be}{\begin{enumerate}}
\newcommand{\ee}{\end{enumerate}}
\newcommand{\tto}[1]{\stackrel{#1}{\longrightarrow}}

\newcommand{\infplus}{\oplus_\infty}
\newcommand{\ver}{\mathrm{vert}}
\newcommand{\G}{\mathcal{G}}
\newcommand{\Aut}{\mathrm{Aut}}
\newcommand{\I}{\textnormal{(i)}}
\newcommand{\II}{\textnormal{(ii)}}

\newcommand{\D}{\mathcal{D}}
\newcommand{\Z}{\mathbb{Z}}
\newcommand{\hooklongrightarrow}{\lhook\joinrel\longrightarrow}
\newcommand{\QM}{\mathrm{QM}}
\newcommand{\deff}{\mathrm{def}\,}
\newcommand{\bbb}{\mathrm{b}}

\newcommand{\thismonth}{\ifcase\month % case 0 --- impossible!
  \or January\or February\or March\or April\or May\or June%
  \or July\or August\or September\or October\or November%
  \or December\fi}

\newcommand{\FF}{{\mathbb F}}

\newcommand{\NN}{{\mathbb N}}

\newcommand{\RR}{{\mathbb R}}

\newcommand{\ZZ}{{\mathbb Z}}

\def\h{{\rm H}}
\def\hb{{\rm H}_{\rm b}}
\def\ehb{{\rm EH}_{\rm b}}

\def\cb{{\rm C}_{\rm b}}

\def\ch{{\rm C}}

\def\one{\mathbf{1\kern-1.6mm 1}}

\def\C{{\operatorname{C}}}
\def\c{{\operatorname{c}}}

\def\h2{{\operatorname{H_2}}}
\def\h1{{\operatorname{H_1}}}

\def\hb{{\rm H}_{\rm b}}
\def\cb{{\rm C}_{\rm b}}
\def\ehb{{\rm EH}_{\rm b}}

\def\h{{\rm H}}

\renewcommand{\phi}{\varphi}

\def\No{N\raise4pt\hbox{\tiny o}\kern+.2em}
\def\no{n\raise4pt\hbox{\tiny o}\kern+.2em}

\newcommand{\Ga}{\Gamma}
\newcommand{\ga}{\gamma}

\newcommand{\R}{\textup{Hom}}
\newcommand{\Ker}{\textup{Ker}}
\newcommand{\Imm}{\textup{Im}}

\begin{document}

\title[Relative second bounded cohomology of free groups]
{Relative second bounded cohomology of free groups}

\author[C. Pagliantini]{Cristina Pagliantini}
\address{Department Mathematik, ETH Z\"urich, 
R\"amistrasse 101, CH-8092 Z\"urich, Switzerland}
\email{cristina.pagliantini@math.ethz.ch}

\author[P. Rolli]{Pascal Rolli}
\address{Department Mathematik, ETH Z\"urich, 
R\"amistrasse 101, CH-8092 Z\"urich, Switzerland}
\email{pascal.rolli@math.ethz.ch}

%\thanks{Both authors were supported by Swiss National Science Foundation project 144373, moreover the second author received %support by project 127016}


\subjclass[2010]{}

%\date{\today}

\begin{abstract} 
This paper is devoted to the computation of the space $\hb^2(\Ga,H;\RR)$,
where $\Ga$ is a free group of finite rank $n\geq 2$ and $H$ 
is a subgroup of finite rank. 
More precisely we prove that  $H$ has infinite index in $\Ga$ if and only if
$\hb^2(\Ga,H;\RR)$ is not trivial, and furthermore,
if and only if
there is an isometric embedding 
$\oplus_\infty^n\mathcal{D}(\ZZ)\hookrightarrow \hb^2(\Ga,H;\RR)$, where $\mathcal{D}(\ZZ)$ is the space
of bounded alternating functions on $\ZZ$ equipped with the
defect norm.
\end{abstract}

\maketitle
%
%
%\tableofcontents
%
%
%%%%%%%%%%%%%%%%%%%%%%%%%%%%%%%%% Aufbau, Dateien %%%%%%%%%%%%%%%%%%%%%%%%%%%
\section{Introduction}\label{sec:intro}

The theory of \emph{bounded cohomology} of groups and spaces was introduced by
Gromov in his seminal paper ``Volume and bounded cohomology'' \cite{Gromov}.
Even if the definition of bounded cohomology is related to the one of singular cohomology, these two theories enjoy totally
different properties. For example, a remarkable
difference is that the bounded cohomology
of a space coincides with the bounded cohomology of its fundamental group \cite{Gromov}.
Another distinctive characteristic  is that it
vanishes in positive degree if the group is amenable
(or if the space has amenable fundamental group).
Later on, Ivanov provided an algebraic foundation of bounded cohomology of 
discrete groups by means of suitable resolutions \cite{Ivanov}, building
on earlier work of Brooks. Since then, computations 
in bounded cohomology have profited from the use of convenient
tools in homological algebra.
More recently, Burger and Monod have developed a functorial approach
to the continuous bounded cohomology of topological groups \cite{BM, BM2, Monod_book}.

\bigskip

For a long time the theory of \emph{quasimorphisms} has been extensively exploited
for studying the second bounded cohomology of a group.
A quasimorphism on a group $\Gamma$ is a map $f:\Gamma\rightarrow \RR$
such that 
$$\sup_{g,h\in \Gamma}|f(gh)-f(h)-f(g)|<\infty.$$
We denote by $\QM(\Ga)$  the $\RR$-vector space of quasimorphisms.
It turns out that the coboundary of a quasimorphism
is a bounded $2$-cocycle and therefore there is a 
linear map 
$$\QM(\Ga)\longrightarrow \hb^2(\Ga;\RR).$$
Moreover, the image of this map is the kernel of the comparison map
$\hb^2(\Ga,\RR)\rightarrow \h^2(\Ga;\RR)$
induced by the inclusion of bounded cochains into ordinary cochains 
(see sections \ref{sec:bounded cohomology} and \ref{sec:qm} for precise definitions and details).

The folklore example of a quasimorphism is Brooks' counting quasimorphism $C_w:\FF_2\rightarrow \RR$ on the rank two free group \cite{Brooks}, for which the parameter $w$ is a word in $\FF_2$. For $g\in \FF_2$ the value of $C_w(g)$ is given by counting  the number of subwords of $g$
equal to $w$, and subtracting the corresponding number for the inverse word
$w^{-1}$. This construction yields an infinite-dimensional subspace of $\hb^2(\FF_2;\RR)$ as was shown by Mitsumatsu \cite{Mit}.  Counting quasimorphisms have been widely generalized with the most far-reaching result being the work of Bestvina-Bromberg-Fujiwara~\cite{BBF} (see also Hull-Osin \cite{HullOsin}). For more background on quasimorphisms we refer the reader to the introduction in \cite{Rolli_thesis} and also to \cite{BM},\cite{Calegari}.

\bigskip

In the relative case Gromov provided a definition of bounded cohomology of 
pairs of topological spaces and pairs of groups. Ivanov's approach seems to run into several difficulties 
in the context of pairs of  spaces and groups. 
A first attempt to extend Ivanov's definition to the relative
case was given by Park \cite{park}.
Park's mapping cone
construction and Gromov's definition determine isomorphic cohomology theories.
However, Frigerio and Pagliantini \cite{FriPa}
showed that Park's semi-norm does not
coincide with the Gromov semi-norm in general. The Gromov semi-norm is the dual of the $\ell^1$-semi-norm
on singular homology that is in turn involved in the definition of simplicial
volume \cite{Loeh} which is one of the geometric applications of bounded cohomology. 
Using Gromov's definition Frigerio and Pagliantini \cite{FriPa}
 provided a (partial) relative version of Gromov's and Ivanov's 
results.

\bigskip

At the best of our knowledge, only few attempts have been made in detecting
the behavior of relative bounded cohomology and results are available only in very few cases. 
In the present paper we focus on the study of the second bounded cohomology 
of a free group of finite rank relative to a subgroup of finite rank.
We introduce relative quasimorphisms and we extend the usual relation between
bounded cohomology and quasimorphisms to the relative setting.
In particular, using the class of split quasimorphisms 
introduced by Rolli \cite{Rolli2},
we show the following characterization
which is the main result of our paper:
\begin{thm}\label{thm_main:intro}
Let $\Ga$ be  a free group of finite rank $n\geq 2$ and let $H<\Gamma$ be 
a subgroup of finite rank. The following are equivalent:
\begin{enumerate}
	\item $H$ has infinite index in $\Gamma$.
        \item The space $\hb^2(\Ga,H;\RR)$ is non-trivial.
	\item The space $\hb^2(\Ga,H;\RR)$ is infinite dimensional as vector space.
        \item There exists a linear isometric embedding 
          $$\oplus_{\infty}^n\mathcal{D}(\ZZ)\hookrightarrow \hb^2(\Ga,H;\RR).$$
\end{enumerate}
\end{thm}
We denote by $\mathcal{D}(\ZZ)$ the space of bounded alternating functions 
on $\ZZ$ equipped with the defect norm (see Section~\ref{sec:qm}).
This is a non-separable Banach space (see \cite[Corollary 7.2]{Rolli_thesis}).
Therefore, the conditions of Theorem~\ref{thm_main:intro} are further equivalent to 
$\hb^2(\Ga,H;\RR)$ being non-separable.
%In particular, the fact that $\hb^2(\Ga,H;\RR)$ is non separable is equivalent to the
%characterizations listed in Theorem \ref{thm_main:intro}.

The crucial step in the proof is the construction of a 
suitable basis for $\Gamma$, namely  a basis that admits split 
quasimorphisms which vanish on the subgroup $H$.
Since this fact is of independent interest from bounded cohomology we formulate here the precise
statement:
\begin{lemma}\label{lem-rel-h2b:intro} 
Let $\Gamma$ be a free group of finite rank $n\geq2$ and let
 $H<\Gamma$ be a subgroup of finite rank and infinite index. 
There exists a basis $\{x_1,\dots,x_n\}$ of $\Gamma$ such that 
for all $g\in\Gamma$ and for all $i$ we have
\[
gHg^{-1}\cap\langle x_i\rangle=\{1\},
\]
which is to say that no conjugate of $H$ contains 
a power of an element of this basis.
\end{lemma}

\medskip

From \cite[Theorem A.1 and Proposition 6.1]{FriPoSi} it follows that the second bounded cohomology of a free group of finite rank $n\geq 2$ relative to a 
 malnormal subgroup of finite rank is infinite dimensional. Indeed, let $\Ga$ be a free group 
of rank $n\geq 2$ and $H$ be a malnormal subgroup of finite rank.
By a result of Bowditch \cite{Bo}  a malnormal quasi convex subgroup of a hyperbolic group  is hyperbolically embedded (see \cite[Definition~2.1]{DGO})
into the hyperbolic group relative to a finite set of generators.
Since all subgroups of a free group are quasi convex, $H$ is hyperbolically embedded in $\Gamma$. Therefore Theorem A.1  in \cite{FriPoSi} implies that
there exists a free group on two generators $\FF_2$ such that $H\cup \FF_2$ is
hyperbolically embedded in $\Gamma$. 
The restriction defines the maps in bounded cohomology 
$$\widehat{\eta}:\ehb^2(\Gamma;\RR)\longrightarrow \ehb^2(H;\RR)$$
and
$$\eta:\ehb^2(\Gamma;\RR)\longrightarrow \ehb^2(H;\RR)\oplus \ehb^2(\FF_2;\RR)$$
where $\ehb^2(\Gamma;\RR)$ is the kernel of the comparison map
$\hb^2(\Gamma;\RR)\rightarrow \h^2(\Gamma;\RR)$
between bounded cohomology and ordinary cohomology (defined in Subsection~\ref{sub:comparison}).
These maps fit in the following commutative diagram
   \begin{center}
      \begin{tikzpicture}
        \def\x{4.5}
        \def\y{-1.2}
        \node (A0_1) at (0*\x, 1*\y) {$\ehb^2(\Gamma;\RR)$};
        \node (A1_1) at (1*\x, 1*\y) {$ \ehb^2(H;\RR)$};
        \node (A1_0) at (1*\x, 0*\y) {$\ehb^2(H;\RR)\oplus \ehb^2(\FF_2;\RR)$};
        \path (A0_1) edge [->] node [auto] {$\eta$} (A1_0);
        \path (A1_0) edge [->] node [auto] {} (A1_1);
        \path (A0_1) edge [->] node [auto] {$\widehat{\eta}$} (A1_1);
      \end{tikzpicture}
    \end{center}
where the vertical arrow is the projection onto the first summand. 
Since $\ehb^2(\FF_2;\RR)$ is infinite dimensional \cite{Brooks,Mit}, 
and the map $\eta$ is surjective by \cite[Proposition 6.1]{FriPoSi}, 
the kernel
of the restriction map $\widehat{\eta}$ is 
infinite dimensional.
Now the conclusion follows by looking at the following segment of exact sequence
$$
0\longrightarrow \hb^2(\Ga,H;\RR)\stackrel{i_b}{\longrightarrow} 
\hb^2(\Ga;\RR) \longrightarrow \hb^2(H;\RR).
$$

Our result generalizes this fact for all (possibly not malnormal) finitely generated subgroups of infinite index 
and moreover provides a complete characterization of which subgroups lead to a non trivial
relative second bounded cohomology.

\subsection{Structure of the paper}
Section \ref{sec:bounded cohomology} is devoted to recalling the definition of relative bounded cohomology,
and in particular to proving some basic facts about relative second bounded cohomology.
In the third section we introduce the notion of (relative) quasimorphisms
showing the relation between (relative) quasimorphisms and
(relative) second bounded cohomology. 
Moreover, we present the special class of split quasimorphisms. Finally,
Section \ref{sec:thm_main} is devoted
to the proof of our main result Theorem~\ref{thm_main:intro} and in particular of
Lemma~\ref{lem-rel-h2b:intro}.

\begin{ack}
We would like to thank Alessandro Sisto for pointing out how the results in \cite{FriPoSi} imply a weak version
of our main theorem.
Both authors were supported by Swiss National Science Foundation project 144373, moreover the second author received support by project 127016.
\end{ack}

\section{Relative bounded cohomology}~\label{sec:bounded cohomology}
This section is devoted to the introduction of definitions and results on
bounded cohomology of a group
and a pair of groups. We refer the reader to
\cite{Gromov,Ivanov, Monod_book} for full details on
bounded cohomology of a group
and to \cite{FriPa, Pa_thesis} for the relative case.
 
Let $\Gamma$ be a discrete group. Recall that
a Banach space $V$ equipped with a linear isometric action of the group $\Ga$
is called a \emph{Banach $\Ga$-module}.

The group cohomology $\h^\ast(\Ga;\RR)$ is computed
by the bar complex $(\ch^n(\Ga), d^n)$ 
$$
0\longrightarrow \RR\tto{d^0} \ch^1(\Ga)\tto{d^1} \ch^2(\Ga)\tto{d^2}\ldots
\tto{d^{n-1}} \ch^n(\Ga)\tto{d^{n}}\ldots
$$
where $\ch^n(\Ga)$ is the space of real maps
on $\Ga^n$
with coboundary operator
$d^0(t)(\ga)=t$ for every $\gamma\in \Ga$ and $t \in \RR$, and
\begin{equation}\label{eq:coboundary}
\begin{array}{rl}
d^n(f)(\ga_1,\ldots,\ga_{n+1})=&
f(\ga_2,\ldots,\ga_{n+1})\\
&+ \sum_{i=1}^{n} (-1)^{i}
f(\ga_1,\ldots,\ga_i\ga_{i+1},\ldots,\ga_{n+1})\\
&+(-1)^{n+1}f(\ga_1,\ldots, \ga_{n})
\end{array}
\end{equation}
for every $n\geq 1$, $f\in \ch^n(\Ga)$, $(\gamma_1,\dots,\gamma_{n+1})\in \Ga^{n+1}$.

For every $f \in \ch^n(\Gamma)$ we set
$$\|f\|_\infty=\sup\{|f(\ga_1,\dots,\ga_n)|\,|\,(\ga_1,\dots,\ga_n)\in\Gamma^{n}\}
\,\in[0,\infty].$$
We denote by $(\cb^n(\Ga),d^n)$ the subcomplex of bounded maps, and 
we turn $\cb^n(\Ga)$ into a Banach $\Ga$-module by the action $\ga\cdot 
f(\ga_1,\dots,\ga_{n})=f(\ga_1,\dots,\ga_{n}\ga)$. 

The \emph{bounded cohomology} module $\hb^n(\Ga;\RR)$ of $\Ga$ is the cohomology of the complex $(\cb^n(\Ga),d^n)$ if $n>0$, otherwise 
$\hb^0(\Ga;\RR)\cong \RR$.
The supremum norm $\|\cdot\|_\infty$ induces a semi-norm on bounded cohomology 
that is usually referred to as
\emph{Gromov semi-norm} and it is still denoted by $\|\cdot\|_\infty$:
$$
\|\beta\|_\infty\colon =\inf_{\xi}\{\|\xi\|_\infty\, | \, \xi\in \cb^n(\Ga), \, d^n\xi=0 , \, [\xi]=\beta \}\qquad 
\text{for}\, \beta \in \hb^n(\Ga;\RR).
$$

\medskip 

Let now $H$ be a subgroup of $\Ga$. As before we consider the $H$-complex
$(\cb^\ast(H),d^\ast)$ which computes the bounded cohomology of $H$. 
The kernel of the
obvious restriction map $\cb^\ast(\Ga)\rightarrow \cb^\ast(H)$ is denoted by 
$\cb^\ast(\Ga,H)$, and we have the short exact sequence of complexes
$$0\longrightarrow \cb^\ast(\Ga,H)\longrightarrow \cb^\ast(\Ga)
\longrightarrow \cb^\ast(H)\longrightarrow 0,
$$
which induces  a long exact sequence in cohomology
$$
\dots\longrightarrow \hb^{n-1}(H;\RR)\longrightarrow \h^n(\cb^\ast(\Ga,H)) \longrightarrow \hb^n(\Ga;\RR)\longrightarrow 
\hb^n(H;\RR)\longrightarrow \dots
$$
The module $\h^n(\cb^\ast(\Ga,H))$ is the $n$-th bounded cohomology 
 of the pair
$(\Ga,H)$ denoted as $\hb^n(\Ga,H;\RR)$.
As in the absolute case the supremum norm on $(\cb^\ast(\Ga,H),d^\ast)$ 
induces a semi-norm on relative bounded cohomology.

\begin{prop}\label{vanishing}
For $(\Ga,H)$ any pair of groups,  we have: 
$$\hb^1(\Ga;\RR)=0 \quad \text{and}
\quad \hb^1(\Ga,H;\RR)=0.$$
\end{prop}
\begin{proof}
A bounded $1$-cocycle on $\Ga$ is a bounded 
homomorphism from $\Ga$ to $\RR$, so it is clear that there are no 
bounded $1$-cocycles. Moreover, a relative bounded $1$-cocycle is in particular an absolute bounded $1$-cocycle.
\end{proof}

\subsection{Dimension $2$} 
In the absolute case the semi-norm induced
on the second bounded cohomology is actually a real norm \cite{Ivanov2,MatMo}. 
In the relative case again we have an actual norm on cohomology:

\begin{prop}\label{embedding}
Let $(\Ga,H)$ be a pair of groups. We have:
\begin{enumerate}
\item The space $\hb^2(\Ga,H;\RR)$ is a Banach space.
\item The map
$$i_b:\hb^2(\Ga,H;\RR)\longrightarrow \hb^2(\Ga;\RR)$$
in the above long exact sequence is a norm 
non-increasing embedding.
\end{enumerate}
\end{prop}
\begin{proof}
The above long exact sequence contains the segment
$$
\hb^1(H;\RR) \longrightarrow \hb^2(\Ga,H;\RR)\stackrel{i_b}{\longrightarrow} 
\hb^2(\Ga;\RR).
$$
By Proposition \ref{vanishing} we have 
$\hb^1(H;\RR)=0$, so that $i_b$ is an embedding.
The inclusion map $\cb^2(\Ga,H)\hookrightarrow \cb^2(\Ga)$
on the cochains level is an isometry, 
and
therefore the induced map $i_b$ in cohomology is norm non-increasing. This proves
part (ii) of the proposition. Now since $i_b$ is a norm non-increasing embedding, and since $\hb^2(\Ga,\RR)$ is a Banach space, 
the semi-norm on $\hb^2(\Ga,H;\RR)$ is a norm. This
means that the latter space is Banach as well.
\end{proof}

\begin{prop}\label{rel-f-ind}
Let $(\Ga,H)$ be a pair of groups. If $H$
has finite index in $\Ga$ then
$\hb^2(\Ga,H;\RR)=0$.
\end{prop}
\begin{proof}
The natural map $\hb^2(\Ga;\RR)\longrightarrow \hb^2(H;\RR)$ is isometrically injective 
by \cite[Proposition 8.6.2]{Monod_book}.
Furthermore we have $\hb^1(H;\RR)=0$ by Proposition \ref{vanishing}. The conclusion follows as we have the following segment of the above long exact sequence 
$$
\hb^1(H;\RR) \longrightarrow \hb^2(\Ga,H;\RR)\stackrel{i_b}{\longrightarrow} 
\hb^2(\Ga;\RR) \longrightarrow \hb^2(H;\RR).
$$
\end{proof}

\subsection{Comparison map}~\label{sub:comparison}
There is a natural 
inclusion map $\c:\cb^n(\Ga)\hookrightarrow \C^n(\Ga)$
(resp.~$\c:\cb^n(\Ga,H)\hookrightarrow \C^n(\Ga,H)$)
which induces the so-called \emph{comparison map}
$$\c:\hb^n(\Ga;\RR)\rightarrow \h^n(\Ga;\RR)
$$ 
(resp.~$\c:\hb^n(\Ga,H;\RR)\rightarrow \h^n(\Ga,H;\RR)$).
This map is a priori neither injective nor surjective.

The behaviour of the comparison map is related to several geometric properties:
Gromov \cite{Gromov} showed that for the fundamental
group $\Ga$ of a closed manifold of negative curvature and $n\geq 2$
the map $\c:\hb^n(\Ga;\RR)\rightarrow \h^n(\Ga;\RR)$
is surjective. 
He also claimed the same surjectivity result for hyperbolic groups
\cite[8.3.T]{Gromov2}.
Using a different approach Mineyev extended this result for any coefficients:
\begin{thm}[{\cite[Theorem 11]{Min_GAFA}}]
Let $\Ga$ be a hyperbolic group and let $V$ be a Banach $\Ga$-module.
Then the map $\c:\hb^n(\Ga;V)\rightarrow \h^n(\Ga;V)$ is surjective 
for each $n\geq 2$. 
\end{thm}
See \cite{Monod_book} for the definition of 
bounded cohomology with coefficients in a Banach $\Gamma$-module.

Later on, Mineyev \cite{Min} 
characterized  hyperbolic groups in terms of the
surjectivity of the comparison map. More precisely,
he showed that a finitely presented group $\Gamma$
is hyperbolic if and only if the comparison map
$\hb^2(\Gamma;V)\rightarrow \h^2(\Gamma;V)$ is
surjective for any Banach $\Gamma$-module $V$.

Due to the duality between the  supremum semi-norm in bounded cohomology and the $\ell^1$-semi-norm
in singular homology, 
the surjectivity of the previous comparison map implies the positivity of the \emph{simplicial volume}. More precisely,
let $M$ be a closed connected oriented manifold of dimension at least $2$ that is rationally essential (e.g.,~aspherical). If the fundamental group of $M$ is hyperbolic then the simplicial volume of $M$ is positive
(see \cite{Loeh} for a survey on simplicial volume).

Moreover, the injectivity of the comparison map is equivalent to 
the so-called \emph{uniform bounded
condition} introduced by Matsumoto and Morita \cite[Theorem 2.8]{MatMo}: 
a normed chain complex $(C, \partial ; \|\cdot\|)$ satisfies the uniform
bounded condition in degree $q$, for $q\in \NN$, if there is a constant $K\in \RR_{>0}$ such that for any
null-homologous $q$-cycle $z$ there exists  a $(q+1)$-chain $b$ with
$$\partial b=z\ , \qquad \|b\|\leq K\|z\|.
$$

In dimension two the behaviour of the comparison map is strictly related to
the notion of quasimorphisms. We extensively explain this relation in the following section.

\section{Quasimorphisms and relative quasimorphisms}\label{sec:qm}
In this section we recall the notion of quasimorphism on a group and its natural extension for a pair of groups. 
We investigate the relation between quasimorphisms and the kernel of the comparison map
between bounded cohomology and ordinary cohomology.

\subsection{Quasimorphisms}
Let $\Gamma$ be a discrete group.
\begin{defi}
A map $f:\Gamma\rightarrow \RR$ is called \emph{quasimorphism}
if there exists a constant $C>0$ such that
$$|f(gh)-f(g)-f(h)|<C, \qquad \forall g,h \in \Gamma.$$
We denote by $\QM(\Ga)$ the $\RR$-vector space of quasimorphisms on $\Ga$.
\end{defi}
\begin{defi}
The \emph{defect} of a quasimorphism $f$ is defined to be:
$$
\deff f \colon =\sup_{g,h\in \Ga} |f(gh)-f(g)-f(h)|.
$$
\end{defi}

\begin{defi}
A quasimorphism $f:\Ga\rightarrow \RR$ is called \emph{trivial}
if there exists a group homomorphism $\xi:\Ga\rightarrow \RR$ such that
$$\sup_{h\in\Ga}|f(h)-\xi(h)|<\infty.$$
\end{defi}
We denote the subspace of $\QM(\Ga)$ 
of all trivial quasimorphisms on $\Ga$ as
$\QM_0(\Ga)$, which is described by
$\cb^1(\Ga)\oplus\R(\Ga,\RR)$. 

The $1$-coboundary (as defined in \eqref{eq:coboundary}) of a quasimorphism $f$ 
is given by
$d^1f(g,h)=f(g)+f(h)-f(gh)$ so $d^1f$ is a bounded $2$-cocycle.
We denote by $\omega_f\colon=[d^1 f]_b$ the corresponding bounded
cohomology class, then we have a linear map
$$
\begin{array}{rcl}
\QM(\Ga)&\longrightarrow &\hb^2(\Ga;\mathbb{R})\\
f&\longmapsto & \omega_f
\end{array}
$$
such that the following sequence
$$\QM(\Ga)\longrightarrow \hb^2(\Ga;\mathbb{R})\stackrel{\c}{\longrightarrow} \h^2(\Ga;\mathbb{R})$$
is exact.
In particular, there exists a canonical isomorphism 
$$\ehb^2(\Gamma;\RR)\cong \QM(\Ga)/\QM_0(\Ga)$$
where $\ehb^2(\Gamma;\RR)$ denotes the kernel
of the comparison map in dimension two. 
The proof is straightforward, and in Proposition \ref{prop:qm_bc}
we will give details for the
analogous statement  for pairs of groups.
As an obvious consequence if $\h^2(\Ga;\RR)=0$ then $\hb^2(\Ga;\RR)\cong \QM(\Ga)/\QM_0(\Ga)$.

\subsection{Split quasimorphisms}
The second author introduced a class of quasimorphisms called \emph{split quasimorphisms} \cite{Rolli, Rolli2}. We recall this construction here.

Let $\Ga$ be a discrete group. We denote by
$$\QM_{\text{alt}}(\Ga)\colon =\{f \in \QM(\Ga)\, |\, f(g)+f(g^{-1})=0 \, \forall g \in \Gamma\}$$
the space of \emph{alternating} quasimorphisms.
Let $\Ga=A\ast B$ be a splitting of $\Ga$. 
If $f_A \in \QM_{\text{alt}}(A)$ and $f_B\in \QM_{\text{alt}}(B)$,
a split quasimorphism is a map
$$f=f_A\ast f_B:\Ga\longrightarrow \RR$$
defined as follows: let $g\in \Ga$ be a non-trivial 
element which has normal form $g=a_1b_1\dots a_nb_n$, 
where $a_i\in A$, $b_i \in B$ and only $a_1$ or $b_n$ are possibly trivial. We set
$$
(f_A\ast f_B)(g)\colon = f_A(a_1)+f_B(b_1)+\dots+f_A(a_n)
+f_B(b_n).
$$ 
Furthermore $(f_A\ast f_B)(1)\colon =0$.

The map $f=f_A\ast f_B$ is an alternating quasimorphism on $\Gamma$ with
$\deff f =\max\{\deff f_A, \deff f_B\}$.

\begin{thm}[{\cite[Theorem 14]{Rolli2}}]\label{defect norm}
Let $f=f_A\ast f_B$ be a split quasimorphism with 
corresponding bounded cohomology class $\omega_f=[d^1f]_b$.
We have
$$
\|\omega_f\|_\infty=\deff f=\max\{\deff f_A,\deff f_B \}.
$$
In particular, $f$ is a minimal defect representative for its class.
\end{thm}
Notice that $f=f_A\ast f_B$ is a minimal defect representative, even if $f_A$ and $f_B$ are not.
In particular, the class of $f_A\ast f_B$ depends on the exact representatives $f_A$ and $f_B$ and not just their classes.

The \emph{defect space} $\mathcal{D}(\Ga)$ of a group $\Ga$ is the space of 
bounded alternating functions $f:\Ga\rightarrow \mathbb{R}$, equipped with the defect 
$\|\cdot\|\colon=\deff \cdot$ as a norm. 
The defect space has the same underlying vector space as the Banach space $\ell^\infty_{\text{alt}}(\Ga)$ endowed with the supremum norm. It turns out that
the defect norm is equivalent to the supremum norm \cite[Corollary B.2]{Rolli2}. Therefore the space $\mathcal{D}(\Ga)$ is a Banach space, and it is non-separable when $\Ga$ has infinitely many elements of order different from $2$. Hence, in particular $\mathcal{D}(\mathbb{Z})$ is non-separable.
For a summary of the basic 
properties of defect spaces we refer to \cite[Appendix B]{Rolli2}, an extended study can be found in \cite{Rolli_thesis}.
Here, we just mention the following property that we will need
in the proof of Theorem~\ref{thm_main:intro}.
\begin{prop}[{\cite[Proposition B.4]{Rolli2}}]\label{prop:B.3}
For a monomorphism $i: H\rightarrow \Gamma$, the map
$$
s_i : \mathcal{D}(H)\longrightarrow \mathcal{D}(\Ga)\, , \qquad s_i(f)(g)=
\left\{
\begin{array}{ll}
f(g) & g=i(h)\\
0 & g\not\in i(H)
\end{array}
\right.
$$
is an isometric embedding.
\end{prop}
The defect norm on the bounded alternating functions allows to have the following isometric embedding:
\begin{thm}[{\cite[Theorem 17]{Rolli2}}]
For a group $\Ga=A\ast B$ there is a linear isometric
embedding 
$$
\mathcal{D}(A)\oplus_\infty \mathcal{D}(B)\hookrightarrow \hb^2(\Ga;\RR)$$
which maps the pair $(f_A,f_B)$ to the bounded
cohomology class $\omega_f$ of the split quasimorphism
$f=f_A\ast f_B$.
\end{thm}
Here the notation $\oplus_\infty$ stands for the direct sum equipped with the max-norm.

\begin{cor}[{\cite[Corollary 19]{Rolli2}}]\label{epi}
If the group $\Gamma$ admits an epimorphism $\Ga\rightarrow \mathbb{F}_2$, then there is a
linear isometric embedding 
$$
\mathcal{D}(\ZZ)\oplus_\infty \mathcal{D}(\ZZ)\hookrightarrow \hb^2(\Ga;\RR).$$
\end{cor}

The construction of split quasimorphisms has an obvious generalization to the case of a free product with several factors, and so do the above results.

\subsection{Relative quasimorphisms}\label{sub:rel}
Let $(\Gamma,H)$ be a pair of discrete groups.
\begin{defi}
A \emph{relative quasimorphism} of the pair $(\Gamma,H)$ is
a quasimorphism on $\Gamma$ which vanishes on the subgroup $H$:
$$\QM(\Ga,H)=\{f\in \QM(\Ga)\, |\, {f_|}_H=0\}.$$
\end{defi}
If $\R(\Ga,H;\RR)$ is the space of homomorphisms on $\Ga$ that vanish on $H$, then the space
of trivial relative quasimorphisms is described as
$\QM_0(\Ga,H)= \cb^1(\Ga,H)\oplus \R(\Ga,H;\RR)$.

As in the absolute case, if $f\in \QM(\Ga,H)\subset \C^1(\Ga,H)$ 
then $d^1f$ is a relative bounded $2$-cocycle and $\omega_f:=[d^1f]_b$ is its relative bounded  cohomology class.
Therefore we have a linear map
$$
\begin{array}{rccl}
\varphi:&\QM(\Ga,H)&\longrightarrow &\hb^2(\Ga,H;\RR)\\
&f&\longmapsto& \omega_f.
\end{array}
$$
\begin{prop}\label{prop:qm_bc}
The sequence
$$
0\longrightarrow \QM_0(\Ga,H)\longrightarrow 
\QM(\Ga,H)\stackrel{\varphi}{\longrightarrow}\hb^2(\Ga,H;\RR)
\stackrel{c}{\longrightarrow}\h^2(\Ga,H;\RR) 
$$
is exact. In particular, there exists a canonical isomorphism 
$$\ehb^2(\Ga,H;\RR)\cong \frac{\QM(\Ga,H)}{\cb^1(\Ga,H)\oplus \R(\Ga,H;\RR)}$$
where  $\ehb^2(\Ga,H;\RR)$ is  the kernel of the relative comparison map in dimension two.
\end{prop}
\begin{proof}
\emph{Exactness in $\QM(\Ga,H)$}. It is straightforward to see that
$\QM_0(\Ga,H)\subseteq \Ker(\varphi)$, and
on the other hand if $f\in \QM(\Ga,H)$ is such that $[d^1 f]_b=0$ 
then there exists $\alpha\in \cb^1(\Ga,H)$
such that $d^1 f=d^1 \alpha$. By definition of the coboundary map in \eqref{eq:coboundary}
the equation $d^1 (f-\alpha)=0$ implies that $f-\alpha$ is a homomorphism.
In particular, since both $f$ and $\alpha$ vanish on $H$ then $f-\alpha\in \R(\Ga,H;\RR)$.

\emph{Exactness in $\hb^2(\Ga,H;\RR)$}. Let  $\omega$ be a coclass in $\Ker(c)$. It 
is represented by $d^1\xi$ for some $\xi\in \ch^1(\Ga,H)$
such that $d^1\xi$ is bounded, so that $\xi\in \QM(\Ga,H)$.
Since $\QM(\Ga,H)\subset \C^1(\Gamma,H)$, the other inclusion $\Imm(\varphi)\subseteq\Ker(c)$ immediately follows.
\end{proof}

\begin{prop}\label{comparison}
If $i:\h^2(\Ga,H;\RR)\rightarrow \h^2(\Ga;\RR)$
is the map induced by the inclusion
of relative cochains into absolute cochains, we have the following isomorphism
$$
\Imm (c)\cap \Ker (i)
\cong\frac{\{f\in \QM(\Ga): {f_|}_{H} \in \R(H;\RR)\}}
{\QM(\Ga,H)\oplus \R(\Ga;\RR)}.
$$
In particular, if $\h^2(\Gamma;\RR)=0$
we can describe the image of the comparison map completely
in terms of relative quasimorphisms.
%$$\hb^2(\Ga,H)\cong \frac{\QM(\Ga,H)}{\C^1(\Ga,H)\oplus \R(\Ga,H;\RR)}\oplus \frac{\{f\in \QM(\Ga): {f_|}_{H} \in \R(H;\RR)\}}
%{\QM(\Ga,H)\oplus \R(\Ga;\RR)}.$$ 
\end{prop}
\begin{proof}
Let us consider the long exact sequence of the pair in ordinary cohomology
$$
\cdots\longrightarrow \h^1(\Ga;\RR)\stackrel{r}{\longrightarrow} \h^1(H;\RR) 
\stackrel{\delta^1}{\longrightarrow}  \h^2(\Ga,H;\RR)\stackrel{i}{\longrightarrow} 
\h^2(\Ga;\RR)\longrightarrow \cdots
$$
For any group $G$ we have $\h^1(G;\RR)=\R(G;\RR)$, then the map $r$ restricts a homomorphism
on $\Ga$ to a homomorphism on $H$. The map $\delta^1$ assigns 
to a homomorphism $f:H\rightarrow \RR$ the coclass $\beta=[d^1 F]$, 
where $F$ 
is any function on $\Ga$ that extends $f$.
We can see $\beta$ as the cohomological obstruction to the existence 
of a homomorphism on $\Ga$ which extends $f$.
Since we restrict to elements
 $\beta \in \Imm(c)\cap \Ker(i)$, up to adding up a coboundary
 we have $\|d^1F\|_\infty<\infty$ which means that $F$
 is a quasimorphism on $\Ga$. Therefore each coclass is represented
by the coboundary of a quasimorphism on $\Ga$ 
extending a homomorphism on the subgroup $H$ and
we have a surjective map:
$$
\begin{array}{rcl}
\{\xi\in \QM(\Ga): {\xi_|}_{H} \in \R(H;\RR)\}&\longrightarrow & 
\Imm(c)\cap \Ker(i)\\
F&\longmapsto&[d^1 F].
\end{array}
$$
It is straightforward to see that 
$\QM(\Ga,H)\oplus \R(\Ga;\RR)$ is contained in the kernel.
Moreover, if $[d^1 F]=0$ there exists $\alpha \in \ch^1(\Ga,H)$
such that $d^1 F=d^1 \alpha$ and, since $F\in \QM(\Ga)$ we have $\|d^1 \alpha\|_\infty<\infty$,
which means that $\alpha\in \QM(\Ga,H)$. 
Finally, by definition of the coboundary map in \eqref{eq:coboundary} 
the equation $d^1 (F-\alpha)=0$ implies that 
$F-\alpha\in \R(\Ga;\RR)$, which concludes the proof.
\end{proof}

\section{Proof of Theorem~\ref{thm_main:intro}}\label{sec:thm_main}
This section is devoted to the proof of Theorem~\ref{thm_main:intro}.
The key to the proof is Lemma \ref{lem-rel-h2b:intro} formulated in the introduction. We recall here the statement:
\begin{lemma}\label{lem-rel-h2b} 
Let $\Gamma$ be a free group of finite rank $n\geq2$ and let
 $H<\Gamma$ be a subgroup of finite rank and infinite index. 
There exists a basis $\{x_1,\dots,x_n\}$ of $\Gamma$ such that 
for all $g\in\Gamma$ and for all $i$ we have
\[
gHg^{-1}\cap\langle x_i\rangle=\{1\},
\]
which is to say that no conjugate of $H$ contains 
a power of an element of this basis.
\end{lemma}
In the proof of Lemma~\ref{lem-rel-h2b} we use the language of Schreier graphs, for which we fix the notation here. Let $\Gamma$ be a free group with a chosen basis $\{x_1,\dots,x_n\}$ and let $H<\Gamma$ be a subgroup. The Schreier graph $\mathcal{G}_H$ of the pair $(\Gamma,H)$ with respect to this basis has as its vertices the set of left cosets
\[
\ver(\G_H)=\{gH\,|\,g\in\Gamma\},
\]
and the edges are given by
\[
\mathrm{edges}(\G_H)=\{(gH,x_igH)\,|\,g\in\Gamma,1\leq i\leq n\}.
\]
Note that each edge is oriented and naturally labelled with a generator $x_i$, and that the graph $\G_H$ is $2n$-regular. Let $gH,g'H\in G/H$ and $x=x_{i_1}^{\epsilon_1}\cdots x_{i_l}^{\epsilon_l}\in\Gamma$, written as a word in the generators and their inverses. We have $xgH=g'H$ if and only if the edge-path $x_{i_1}^{\epsilon_1},\dots,x_{i_l}^{\epsilon_l}$ starting at the vertex $gH$ ends at the vertex $g'H$. For each letter $x_i^{-1}$ with a negative power this path is running in the direction opposite to the orientation of the corresponding edge. In particular we have $xgH=gH$ if and only if the path starting at $gH$ and corresponding to $x$ is a loop. We say that a vertex $gH$ is the basepoint of an $x$-loop if there is $n\geq1$ such that $x^ngH=gH$. We write
\[
L_x(H)\subset\ver(\G_H)
\]
for the set of vertices that are the basepoint of an $x$-loop. For an automorphism $\varphi\in\Aut(\Gamma)$ we have an induced bijection
\[
\varphi:\ver(\G_H)\longrightarrow\ver(\G_{\varphi(H)}),\quad gH\mapsto\varphi(g)\varphi(H)
\]
and by restricting this map we obtain for each $x\in\Gamma$ a bijection
\[
L_x(H)\longrightarrow L_{\varphi(x)}(\varphi(H)).
\]
The graph $\G_H$ can be identified with a covering graph of a wedge of $n$ loops, namely with the covering that corresponds to the given subgroup $H$. Furthermore, $\G_H$ can be seen as the quotient of the Cayley graph $\G_{\{1\}}$ with respect to the natural action of $H$. We denote by $\mathcal{C}_H$ the core of the graph $\G_H$. This is the subgraph that consists of the edges and vertices that are contained in a loop without backtracking. This means that
\begin{equation}\label{eq:core}
v\in\ver(\mathcal{C}_H)\quad\Longleftrightarrow\quad v\in L_x(H) \mbox{ for some cyclically reduced $x\in\Gamma$}
\end{equation}
and in fact $\ver(\mathcal{C}_H)=\bigcup L_x(H)$, where the union is over all cyclically reduced elements of $\Gamma$. Core graphs were introduced by Stallings in \cite{Stallings}, and this is also the reference for the following observations:
\begin{prop}\label{prop-core}
\begin{enumerate}
\item[\I] The graph $\mathcal{C}_H$ is finite, if and only if the subgroup $H$ has finite rank.
\item[\II] We have $\mathcal{C}_H=\G_H$, if and only if the subgroup $H$ has finite index in $\Gamma$.
\end{enumerate}
\end{prop}
In order to prove Lemma~\ref{lem-rel-h2b} we 
fix a basis $\{x_1,\dots,x_n\}$ of $\Gamma$. For $1\leq i\leq n$ and $k\in\Z$ we define $\varphi_{i,k}\in\Aut(\Gamma)$ by
\begin{align*}
\varphi_{i,k}(x_i)&=x_i\\
\varphi_{i,k}(x_j)&=x_jx_i^k,\quad\mbox{for $i\neq j$}.
\end{align*}
We will show that for a suitable concatenation $\psi$ of such automorphisms we obtain a new basis $\{y_1,\dots, y_n\}$ of the desired type by setting $y_i: =\psi^{-1}(x_i)$. 
To this purpose we need three preliminary results.
Let $\mathcal{G}_{H}$ be the Schreier graph of $H$ with respect to the chosen basis and let $\mathcal{C}_{H}$ be its core. We define
\[
L(H):=\bigcup_{i=1}^n L_{x_i}(H)
\]
This is the set of all vertices in $\mathcal{G}_{H}$ at which a loop of a basis element is based. Since each $x_i$ is a cyclically reduced element we have $L(H)\subset\mathrm{vert}(\mathcal{C}_{H})$. In particular, the set $L(H)$ is finite since $H$ has finite rank.
\begin{lemma}\label{Claim I}
If $L(H)\neq\emptyset$ then there exists a vertex $v\in L(H)$ such that~$v\not\in\bigcap_{i=1}^n L_{x_i}(H)$.
\end{lemma}
\begin{proof}
Since $H$ has infinite index the graph $\mathcal{G}_{H}$ has vertices which are not contained in $L(H)$, for example any vertex outside the core. Therefore we can choose $w\in\mathrm{vert}(\mathcal{G}_{H})$ which is not in $L(H)$ but is adjacent to some $v\in L(H)$. If the edge connecting $v$ with $w$ is labelled $x_j$ then either both or none of $v$ and $w$ are contained in an $x_j$-loop. Since $w$ is not in such a loop, $v$ is neither, and therefore $v\not\in L_{x_j}(H)$.
\end{proof}
\begin{lemma}\label{Claim II}
For all $1\leq i\leq n$ there exists $k>0$ such that
\begin{enumerate}
\item[\I]$\forall v\in L_{x_i}(H):\quad x_i^kv=v$
\item[\II]$\forall v\in\mathrm{vert}(\mathcal{C}_{H})\,\backslash\,L_{x_i}(H):\quad x_i^kv\not\in\mathrm{vert}(\mathcal{C}_{H}).$
\end{enumerate}
\end{lemma}
\begin{proof}
We define
\[
k:= N\cdot\mathrm{lcm}\{\ell\,|\,\ell\mbox{ is the length of a simple $x_i$-loop in $\mathcal{G}_{H}$}\}
\]
for a suitable number $N>0$.
Note that the set of which we are taking the least common multiple is finite, since the graph $\mathcal{G}_{H}$ contains only finitely many simple loops. If the set is actually empty (i.e.~if $L_{x_i}(H)=\emptyset$) then we define the lcm to be $1$. The first condition is obviously satisfied for such a $k$. For $v$ as in the second condition, consider the sequence of vertices $v,x_iv,x_i^2v,\dots$. This sequence must eventually leave the core: Otherwise it would contain a loop and this loop would contain $v$, a contradiction since $v$ is not contained in any $x_i$-loop. 
Now choosing $N$ sufficiently large we have that $x_i^kv$ is outside the core for all of the finitely many vertices $v\in\mathrm{vert}(\mathcal{C}_{H})\,\backslash\,L_{x_i}(H)$.
\end{proof}
\begin{lemma}\label{Claim III}
Let $1\leq i\leq n$ and let $k>0$ be the number from Lemma~\ref{Claim II}. For the automorphism $\varphi=\varphi_{i,-k}$ the bijection $\varphi^{-1}=\varphi_{i,k}:\mathrm{vert}(\mathcal{G}_{\varphi(H)})\longrightarrow\mathrm{vert}(\mathcal{G}_{H})$ restricts to a map
\[
L(\varphi(H))\longrightarrow L_{x_i}(H).
\]
\end{lemma}
\begin{proof}
Let $v\in L_{x_j}(\varphi(H))$ for some $1\leq j\leq n$. If $i=j$ then $\varphi^{-1}(v)\in L_{\varphi^{-1}(x_i)}(H)=L_{x_i}(H)$. 
If $i\neq j$ then $\varphi^{-1}(v)\in L_{\varphi^{-1}(x_j)}(H)=L_{x_jx_i^k}(H)$.
We claim that $L_{x_jx_i^k}(H)\subseteq L_{x_i}(H)$, hence $\varphi^{-1}(v)\in L_{x_i}(H)$.
To prove the claim, %recall that 
%\[
%v\in\ver(\mathcal{C}_H)\quad\Longleftrightarrow\quad v\in L_x(H) \mbox{ for some cyclically reduced $x\in\Gamma$}
%\]
let us pick $w \in L_{x_jx_i^k}(H)$. Since $x_jx_i^k$ is a cyclically reduced element then $w\in\ver(\mathcal{C}_H)$ using the identification \eqref{eq:core}. Suppose by contradiction that $w \notin L_{x_i}(H)$. Then by Lemma \ref{Claim II}-(ii) 
$x_i^kw$ is outside the core, %for any vertex $w$ which is not in an $x_i$-loop, 
and in particular $x_jx_i^kw$ is outside the core. This yields a contradiction. %for such a vertex. 
%It follows that $L_{x_jx_i^k}(H)\subset L_{x_i}(H)$, hence $\varphi^{-1}(v)\in L_{x_i}(H)$.
%$x_jx_i^kw$ is outside the core, and in particular not in $L_{x_jx_i^k}(H)$, a contradiction.
\end{proof}
We are now ready to conclude the proof of Lemma~\ref{lem-rel-h2b}.
\begin{proof}[Proof of Lemma~\ref{lem-rel-h2b}]
Lemma~\ref{Claim I} tells us that there is an index $1\leq i\leq n$ such that $L_{x_i}(H)\subsetneq L(H)$. Let $k$ be the associated number from Lemma~\ref{Claim II} and let $\varphi$ be the automorphism from Lemma~\ref{Claim III}. Since the map $L(\varphi(H))\longrightarrow L_{x_i}(H)$ is injective we have
\[
|L(\varphi(H))|\leq |L_{x_i}(H)|<|L(H)|.
\]
By replacing the subgroup $H$ with its image $\varphi(H)$ we reduced the number of vertices that are contained in a basic loop. 
If we iterate this argument we obtain after finitely many steps an automorphism $\psi\in\Aut(\Gamma)$ for which $L(\psi(H))=\emptyset$. So we have $L_{x_i}(\psi(H))=\emptyset$ for all $i$, and this means that
\begin{align*}
\forall i\,\forall g\in\Gamma:\quad g\psi(H)g^{-1}\cap\langle x_i\rangle=\{1\},
\end{align*}
which amounts to saying
\begin{align*}
\forall i\,\forall g\in\Gamma:\quad gH g^{-1}\cap\langle y_i\rangle=\{1\},
\end{align*}
where $y_i=\psi^{-1}(x_i)$.
\end{proof}
\begin{lemma}\label{lem-relative-2} Let $\Gamma$ be a free group of finite rank $n\geq2$ and let $H<\Gamma$ be a subgroup of finite rank and infinite index. Assume that $\{x_1,\dots,x_n\}$ is a basis of $\Gamma$ such that $gHg^{-1}\cap\langle x_i\rangle=\{1\}$ for all $i$ and all $g\in\Gamma$. Then there is a number $m_0>0$ such that no element of $H$ contains a power $x_i^m$ with $m>m_0$ as a subword.
\end{lemma}
\begin{proof} Assume that, under the assumptions of the lemma, we can find infinitely many powers $x_i^m$ as subwords of cyclically reduced elements of $H$. In particular, $H$ contains an element of the cyclically reduced form $yx_i^mz$ for some $m>\#\ver(\mathcal{C}_H)$ and some $i$. (Note that the core $\mathcal{C}_H$ is finite by Proposition~\ref{prop-core}.) The path in the Schreier graph $\G_H$ which starts at the vertex $H$ and is described by this element is then a loop without backtracking based at $H$. This loop is contained in $\mathcal{C}_H$, in particular the following vertices on this loop are contained in $\ver(\mathcal{C}_H)$:
\[
x_izH,\,x_i^2zH,\,\dots,\,x_i^mzH
\]
By the choice of $m$, there exist $1\leq k,l\leq m$, $k\neq l$, such that $x_i^kzH=x_i^lzH$, which means that $x_i^{k-l}\in zHz^{-1}$, a contradiction.
\end{proof}
\noindent Combining the Lemmas~\ref{lem-rel-h2b} and \ref{lem-relative-2} immediately yields
\begin{lemma}\label{lem-final}Let $\Gamma$ be a free group of finite rank $n\geq2$ and let $H<\Gamma$ be a subgroup of finite rank and infinite index. There exists a basis $\{x_1,\dots,x_n\}$ of $\Gamma$ that has the following property: There is a number $m_0>0$ such that no element of $H$ contains a power $x_i^m$ with $m>m_0$ as a subword.
\end{lemma}
\begin{proof}[Proof of Theorem~\ref{thm_main:intro}]
The implications (iv)$\Rightarrow$(iii) and (iii)$\Rightarrow$(ii) are obvious and the implication (ii)$\Rightarrow$(i) follows from Proposition~\ref{rel-f-ind}. 
In order to prove (i)$\Rightarrow$(iv) we let $\{x_1,\dots,x_n\}$ be the basis of $\Gamma$ given by Lemma~\ref{lem-final} and we let $m_0>0$ be the number that is also given by this Lemma. Associated to the splitting $\Gamma=\langle x_1\rangle\ast\langle x_2\rangle\ast\cdots\ast\langle x_n\rangle$ we have the isometric embedding
\[
I:\,\oplus_\infty^n\D(\Z)\hooklongrightarrow\h^2_\bbb(\Gamma,\RR)
\]
from Corollary \ref{epi} (We are using the obvious generalization of the statement to the case of several free factors). By 
Proposition~\ref{prop:B.3} we have an isometric embedding $j:\D(m_0\Z)\hooklongrightarrow\D(\Z)$ from which we obtain the isometric embedding
\[
J:=\oplus^nj:\,\oplus_\infty^n\D(m_0\Z)\hooklongrightarrow\oplus_\infty^n\D(\Z).
\]
By Proposition~\ref{embedding} the relative bounded cohomology $\h^2_\bbb(\Gamma,H;\RR)$ is embedded in $\h^2_\bbb(\Gamma,\RR)$, and we claim that the concatenation $I\circ J$ ranges in this subspace. So let $\{f_i\}=\{f_i\}_{1\leq i\leq n}\in\oplus^n_\infty\D(m_0\Z)$. Then $J(\{f_i\})=\{j(f_i)\}$, and each of the functions $j(f_i)\in\D(\Z)$ is supported on the subgroup $m_0\Z$, by construction of the embedding $j$. Write $f:=j(f_1)\ast\cdots\ast j(f_n)$ for the split quasimorphism obtained from these functions. This quasimorphism has the property that $f(g)=0$ if $g$ is an element whose factorization contains no factor of the form $x_i^{k\cdot m_0}$, $k\neq0$. Since we have chosen $m_0$ according to Lemma~\ref{lem-final}, we know that in the factorization of each element of $H$ only exponents less than $m_0$ occur. Therefore we have $f|_H=0$, i.e.~$f\in\QM(\Gamma,H)$. In particular we have $d^1 f|_{H\times H}=0$, so the class $(I\circ J)(\{f_i\})=[d^1 f]_\bbb$ is in fact contained in $\h^2_\bbb(\Gamma,H;\RR)$. Therefore we have an embedding
\[
\oplus_\infty^n\D(\Z)\cong\oplus_\infty^n\D(m_0\Z)\hooklongrightarrow\h^2_\bbb(\Gamma,H;\RR).
\]
We are left with showing that this embedding is isometric. Let $i_b:\h^2_\bbb(\Gamma,H;\RR)\hooklongrightarrow\h^2_\bbb(\Gamma,\RR)$ be the embedding from Proposition~\ref{embedding}. Let $\{f_i\}\in\infplus^n\D(m_0\Z)$ as before, with associated split quasimorphism $f$. Let $\omega_f\in\h^2_\bbb(\Gamma,H;\RR)$ be the image of $\{f_i\}$ under the embedding in question. We have
\[
\deff f=\|i_b(\omega_f)\|_\infty\leq\|\omega_f\|_\infty\leq\|d^1f\|_\infty=\deff f
\]
where the first equality comes from Theorem \ref{defect norm}, the first inequality is due to the fact that $i_b$ is norm non-increasing (Proposition~\ref{embedding}.(ii)), and the second inequality is because $d^1 f$ is a representative of the relative class $\omega_f$. It follows that
\[
\|\omega_f\|_\infty=\deff f=\max_i\{\deff f_i\}=\|\{f_i\}\|.\qedhere
\]
\end{proof}

\vskip1cm

\providecommand{\bysame}{\leavevmode\hbox to3em{\hrulefill}\thinspace}
\providecommand{\MR}{\relax\ifhmode\unskip\space\fi MR }
\providecommand{\MRhref}[2]{%
  \href{http://www.ams.org/mathscinet-getitem?mr=#1}{#2}
}
\providecommand{\href}[2]{#2}

\end{document}